\documentclass[3p,times]{article}

\usepackage[left=2cm, right=2cm, top=2cm]{geometry}

\usepackage{amsmath,amssymb,amsthm,mathtools,mathdots,nicefrac,xcolor,stmaryrd}
\usepackage{authblk}

\theoremstyle{definition} 
\newtheorem{theorem}{Theorem}[section]

\newtheorem{corollary}[theorem]{Corollary}
\newtheorem{lemma}[theorem]{Lemma}
\newtheorem{proposition}[theorem]{Proposition}

\newtheorem{example}[theorem]{Example}

\newtheorem{question}[theorem]{Question}

\DeclareMathOperator{\val}{val}
\DeclareMathOperator{\tw}{tw}

\DeclareMathOperator{\gon}{gon}

\newcommand{\sn}{\text{sn}}

\DeclarePairedDelimiter\abs{\lvert}{\rvert}

\DeclarePairedDelimiter\norm{\lVert}{\rVert}

\DeclarePairedDelimiter\paren{(}{)}

\makeatletter
\let\oldabs\abs
\def\abs{\@ifstar{\oldabs}{\oldabs*}}
\let\oldnorm\norm
\def\norm{\@ifstar{\oldnorm}{\oldnorm*}}
\let\oldparen\paren
\def\paren{\@ifstar{\oldparen}{\oldparen*}}
\makeatother

\title{On the scramble number of graphs}
\author{Marino Echavarria, Max Everett, Robin Huang, Liza Jacoby, Ralph Morrison, and Ben Weber}

\begin{document}


\maketitle

\begin{abstract}
The scramble number of a graph is an invariant recently developed to aid in the study of divisorial gonality.  In this paper we prove that scramble number is NP-hard to compute, also providing a proof that computing gonality is NP-hard even for simple graphs, as well as for metric graphs.  We also provide general lower bounds for the scramble number of a Cartesian product of graphs, and apply these to compute gonality for many new families of product graphs. 

\end{abstract}

\section{Introduction}

Chip-firing games on graphs form a combinatorial analog of divisor theory on algebraic curves.  Although these games were earlier studied through the lenses of structural combinatorics \cite{chip_firing} and the Abelian sandpile model \cite{sandpile}, it was Baker and Norine \cite{bn} that established the parallels with algebraic geometry, defining divisors, equivalence, degree, and rank and proving a graph-theoretic version of the Riemann-Roch theorem.  This was extended in \cite{gk,mz} from finite\footnote{Throughout when we say \emph{finite}, we mean \emph{not metric}.} graphs to metric graphs, which have lengths assigned to their edges.  Combined with Baker's specialization lemma \cite{baker_specialization}, this framework has been used to study algebraic curves through purely combinatorial means, as in \cite{cdpr}. 

The \emph{divisorial gonality} (or simply \emph{gonality}) of a (finite) graph is the minimum degree of a rank \(1\) divisor on that graph.  Informally, it can be thought of as the minimum number of chips that can be placed on that graph such that \(-1\) debt can be eliminated via chip-firing moves, no matter where that debt is placed.  This is one of several graph theoretic analogs of the gonality of an algebraic curve, the minimum degree of a rational map from the curve to the projective line.  

It was shown in \cite{graph_gonality_is_hard} that the gonality of a graph is NP-hard to compute; indeed, it is in the class of APX-hard problems, which are hard to even approximate.  In specific examples it is often feasible to give an upper bound \(\textrm{gon}(G)\leq d\) on gonality, by exhibiting a positive rank divisor of degree \(d\).  However, to conclude \(\gon(G)=d\) one must  also argue \(\textrm{gon}(G)\geq d\); this is quite difficult with a na\"{i}ve approach, as one would have to check all effective divisors of degree \(d-1\) to verify that they do not have positive rank.  Thus lower bounds on graph gonality are of paramount importance.

Several lower bounds relate gonality to well-studied invariants from graph theory.  It was shown in \cite{ckk} that \(\min(\lambda(G),|V(G)|)\leq \gon(G)\), where \(\lambda(G)\) denotes the edge-connectivity of \(G\).  A series of improvements on this bound came from \cite{db}, which showed that the strict bramble number of a graph is a lower bound on gonality; and from \cite{treewidth}, which showed that the treewidth of a graph is a lower bound on gonality.

The most powerful lower bound known to date is the \emph{scramble number} of a graph, the focus of this paper.  This invariant was introduced in \cite{harp2020new}, where it was proved that the scramble number is at least as large as treewidth, and is no larger than gonality:
\[\textrm{tw}(G)\leq \textrm{sn}(G)\leq\textrm{gon}(G).\]
Scramble number has been used to compute several previously unknown graph gonalities, including for stacked prism graphs and toroidal grid graphs of arbitrary size.  As such, scramble number shows a great deal of promise as a tool for studying divisor theory of graphs, and is worthy of further exploration.

Our first major result regards the computational complexity of the scramble number of a graph.  Starting with a (simple) graph \(G\) on \(m\) vertices, we construct a (simple) graph \(\widehat{G}\) on \(n=2m\) vertices with \(\textrm{sn}(\widehat{G})=\textrm{gon}(\widehat{G})=2m-\alpha(G)\) , where \(\alpha(G)\) denotes the independence number of \(G\). We also prove the same result holds for any metric version \(\Gamma\) of such a finite graph \(G\). Since computing \(\alpha(G)\) is NP-complete, this leads us to the following three results.

\begin{theorem}\label{theorem:NP_hard_scramble}
Computing the scramble number of a graph is NP-hard.
\end{theorem}

\begin{theorem}\label{theorem:NP_hard_gonality}
Computing the gonality of a simple graph is NP-hard.
\end{theorem}

\begin{theorem}\label{theorem:NP_hard_metric}
Computing the gonality of a metric graph is NP-hard.
\end{theorem}

Although it was already shown in \cite{graph_gonality_is_hard} that the gonality of a (non-metric) multigraph is NP-hard to compute, the argument constructed a multigraph \(\widehat{G}\) with many parallel edges, thus leaving open the possibility that gonality on simple graphs could be computationally easier.  Our result closes this possibility.

The other major topics in our paper are the scramble number and the gonality of the Cartesian product \(G\square H\) of two connected graphs \(G\) and \(H\). The gonality of particular product graphs has been studied in \cite{glued_grid_graphs,db}, with a more unified approach taken in \cite{gonality_product}, which provided the following upper bound:
\[\textrm{gon}(G\square H)\leq \min(|V(G)|\gon(H),|V(H)|\gon(G)).\]
We provide several lower bounds on the scramble number of \(G\square H\) in Theorem \ref{theorem:sn_lower_bound_general}, Corollary \ref{corollary:k=1}, and Proposition \ref{proposition:k=2}.  When one of these lower bounds on scramble number equals the upper bound on gonality, these bounds are equal to both scramble number and gonality.  We provide many instances of product graphs where this allows us to determine graph gonality, including the product of any graph with any sufficiently large tree (see Theorem \ref{theorem:product_gon}(i)). Table \ref{table:gonalities} summarizes most of the product graphs whose gonalities we determine in Section \ref{section:applications_to_gonality}.  In this table, \(G\) and \(H\) can be any connected graphs; \(T\) is a tree; \(K_n\) is a complete graph on \(n\) vertices; \(K_{m,n}\) is a complete bipartite graph on \(m\) and \(n\) vertices; and \(C_m\) is a cycle graph on \(m\) vertices.

\begin{table}[hbt]
\begin{center}
 \begin{tabular}{|c c c c |} 
 \hline
 \(G\) & \(H\) & \textrm{Assumptions}&\(\textrm{gon}(G\square H\)) \\
 \hline
  \(T\) & \(H\) & \(\lambda(H)=\gon(H)=k\) & \(\min(|V(H)|,k|V(G)|)\)) \\ 
    \(G\) & \(K_n\square T\) & \(|V(G)|\leq |V(T)|\) & \(n|V(G)|\)
\\ 
   \(G\) & \(H\) & \(\lambda(H)=\gon(H)=2,|V(G)|\leq |V(H)|/2\) & \(2|V(G)|\)
\\ 
   \(G\) & \(K_{m,n}\) & \(|V(G)|\leq (m+n)/m\) & \(|V(G)|\cdot\min(m,n)\)
\\ 
   \(C_m\) & \(K_n\) & \(m,n\geq 2\) & \(\min(2n,m(n-1))\)
\\ 
   \(G\) & \(H\) & \(\kappa(G)=\kappa(H)=\gon(G)=\gon(H)=2\) & \(2\min(|V(G)|,|V(H)|)\)
\\ 
   \(C_\ell\square C_m\) & \(C_n\) & \(\ell,m\geq 3,\frac{2}{3}\ell m\leq n\) & \(2\ell m\)
\\ 
   \(K_k\square T\) & \(K_\ell\) & \(k<\ell\leq k(|T|-2)+4\) & \(k\ell\)
\\ 
 \hline
\end{tabular}
\caption{Graph gonalities computed in Section \ref{section:applications_to_gonality}}
\label{table:gonalities}
\end{center}
\end{table}

Our paper is organized as follows.  In Section \ref{section:preliminaries} we present background and several lemmas on graph theory, scramble number, and chip-firing games and gonality.  In Section \ref{section:complexity} we present our results on computational complexity. In Section \ref{section:lower_bound} we provide our lower bounds on scramble number for product graphs, which we apply to study graph gonality in Section \ref{section:applications_to_gonality}.

\medskip

\noindent \textbf{Acknowledgements:} The authors are grateful for the support they received from the Williams SMALL REU, and from NSF grants DMS-1659037 and DMS-2011743.  They also thank thank Benjamin Bailey, David Jensen, and Noah Speeter for insightful conversations on scramble number, and for helpful comments on drafts of this paper.

\section{Preliminaries}
\label{section:preliminaries}

In this section we establish terminology, notation, and background results.  We start with standard results and definitions from graph theory, including such invariants as connectivity and treewidth.  We then present scrambles on graphs, followed by divisor theory and chip-firing.

\subsection{Graphs}  Throughout this paper by a graph $G$ we mean a pair $(V,E)$ of a finite vertex set $V$ and a finite edge multiset $E$, with multiple edges allowed between vertices although not from a vertex to itself.  In certain sections we will assume that our graphs are connected, and will explicitly state this at the start.  We refer to the number of edges incident to a vertex \(v\) as the \emph{valence}\footnote{This is usually referred to as the \emph{degree} of a vertex; however, we need to reserve the term ``degree'' for divisor theory.} of \(v\), denoted \(\textrm{val}(v)\).  Given disjoint subsets \(A,B\subset V(G)\), we denote by \(E(A,B)\subset E(G)\) the multiset of edges of \(G\) with one vertex in \(A\) and the other vertex in \(B\).

We now recall three possible relations between two graphs:  the subgraph relation, the minor relation, and the immersion minor relation.  We say that \(H\) is a \emph{subgraph} of \(G\) if \(H\) can be obtained from \(G\) by deleting vertices and edges.  Given a subset \(S\subset V(G)\), the subgraph \(H\) \emph{induced} by \(S\) is the graph with vertex set \(S\) and all edges in \(E(G)\) connecting vertices in \(S\).  If the subgraph induced by \(S\) is connected, we say that \(S\) is a \emph{connected subset} of \(V(G)\).

We say that \(H\) is a \emph{minor} of \(G\) if \(H\) can be obtained from \(G\) by deleting vertices, deleting edges, and contracting edges.  A family of graphs \(\mathcal{G}\) is called \emph{minor closed} if any minor of a graph in \(\mathcal{G}\) is also in \(\mathcal{G}\); and a graph invariant \(f\) is called \emph{minor monotone} if \(f(H)\leq f(G)\) whenever \(H\) is a minor of \(G\).

For the immersion minor relation, we first need to define lifts of graphs.  If \(u,v,w\in V(G)\) where \(u,v,w\) are all distinct and \(uv,vw\in E(G)\), the \emph{lift of \(G\) at \(v\) with respect to \(u\) and \(w\)} is the graph obtained by deleting the edges \(uv\) and \(vw\) and adding an edge \(uw\).  Since our graphs need not be simple, we remark that there could be multiple edges between \(u\) and \(v\), or multiple edges between \(v\) and \(w\); in this case only one copy of each edge is deleted. If \(H\) can be obtained from \(G\) by a taking a subgraph and performing a sequence of lifts, we say \(H\) is an \emph{immersion minor} of \(G\). A family of graphs \(\mathcal{G}\) is called \emph{immersion minor closed} if any immersion minor of a graph in \(\mathcal{G}\) is also in \(G\); and a graph invariant \(f\) is called \emph{immersion minor monotone} if \(f(H)\leq f(G)\) whenever \(H\) is an immersion minor of \(G\).

We now recall several graph invariants.  Given a graph $G$ with vertex set $V(G)$ and edge set $E(G)$, a \emph{vertex cut} of $G$ is a subset $S \subsetneq V(G)$ such that the removal of the vertices in $S$ yields at least two connected components in the resulting graph $G'$ (or the graph on a single vertex). Similarly an \emph{edge cut} $D \subset E(G)$ is a multiset for which removing the edges in $D$ results in a disconnected graph (or the graph on a single vertex). The \emph{vertex-connectivity} (or simply the \emph{connectivity}) of a graph is the smallest cardinality of a vertex cut, and is denoted \(\kappa(G)\); and similarly the edge-connectivity of a graph is the smallest cardinality of an edge cut, and is denoted \(\lambda(G)\). We remark that \(\kappa(G)\leq\lambda(G)\), and that if \(G\) is not connected then \(\kappa(G)=\lambda(G)=0\).  We say a graph \(G\) is \emph{\(k\)-vertex connected} if \(\kappa(G)\geq k\), and that \(G\) is \emph{\(k\)-edge connected} if \(\lambda(G)\geq k\).  A key result about the two versions of connectivity is the following.

\begin{theorem}[Menger's Theorem, \cite{Menger1927}]\label{theorem:mengers}
A graph is $k$-vertex connected if and only if every pair of vertices has at least $k$ vertex-disjoint paths between them.  A graph is $k$-edge connected if and only if every pair of vertices has $k$ edge-disjoint paths between them.
\end{theorem}

Another well-studied graph invariant for connected graphs is the \emph{treewidth} of a graph\footnote{For a disconnected graph, the treewidth is the maximum treewidth of any connected component.}.  Treewidth is the minimum width of any tree decomposition of the graph, as defined in \cite{treewidth_original}. We can also define treewidth in terms of brambles.  A \emph{bramble} on a graph is a collection \(\mathcal{B}\) of subsets \(B\subset V(G)\) such that each \(B\) is connected, as is \(B\cup B'\) for any \(B,B'\in \mathcal{B}\).  A set \(C\subset V(G)\) is called a \emph{hitting set} for \(\mathcal{B}\) if \(C\cap B\neq \emptyset\) for all \(B\in \mathcal{B}\); and the \emph{order} of \(\mathcal{B}\), denoted \(||\mathcal{B}||\) is the minimum size of a hitting set.  The treewidth of a graph is then one less than the maximum order of a bramble by \cite{seymour-thomas}.

Treewidth is a minor monotone function, so that if \(H\) is a minor of \(G\) then we have \(\tw(H)\leq \tw(G)\).  The set of all graphs of treewidth at most \(k\) is thus a minor-closed family, and admits a characterization by a finite forbidden set of minors.  For instance, graphs of treewidth at most \(2\) have the complete graph \(K_4\) as the unique forbidden minor.

We close this subsection by recalling a graph operation.  Given two graphs \(G\) and \(H\), the \emph{Cartesian product} \(G\square H\) is the graph with vertex set \(V(G)\times V(H)\), where \((u_1,v_1)\) is incident to \((u_2,v_2)\) with \(e\) edges if and only if \(u_1=u_2\) and \(v_1\) and \(v_2\) are connected by \(e\) edges, or  \(v_1=v_2\) and \(u_1\) and \(u_2\) are connected by \(e\) edges.  Given a represetation of the a graph as \(G\square H\), we refer to subgraphs of the form \(u\square H\) as  \emph{canonical copies of \(H\)}, to subgraphs of the form \(G\square v\) as  \emph{canonical copies of \(G\)}.  Note that there are \(|V(G)|\) canonical copies of \(H\), and \(|V(H)|\) canonical copies of \(G\), even though there may be more subgraphs of \(G\square H\) isomorphic to \(G\) or to \(H\).

\subsection{Scrambles}

We now present the definition of and background on scrambles, which were introduced in \cite{harp2020new} as a generalization of brambles.  A \emph{scramble}  \(\mathcal{S}=\{E_1,\ldots,E_s\}\) on a graph \(G\) is a collection of connected (and nonempty) subsets \(E_i\) of \(G\). The elements of \(\mathcal{S}\) are called \emph{eggs}. The \emph{scramble order} of a scramble \(\mathcal{S}\) is the maximum integer \(k\) such that:
\begin{itemize}
    \item[(i)] no set \(C\) of size less than \(k\) covers \(\mathcal{S}\), and
    \item[(ii)] if \(A\subset V(G)\)   such that there exist eggs \(E_i,E_j\in \mathcal{S}\) with \(E_i\subset A\) and \(E_j\subset A^C\), then \(|E(A,A^C)|\geq k\).
\end{itemize}
The scramble order of \(\mathcal{S}\) is written \(||\mathcal{S}||\).

As with brambles, a set \(C\) covering \(\mathcal{S}\) is called a hitting set for \(\mathcal{S}\). We call a set \(A\) satisfying condition (ii) an \emph{egg-cut} of \(\mathcal{S}\), and we refer to \(|E(A,A^C)|\) as the \emph{size} of the egg-cut.  We then have that \(||\mathcal{S}||=\min(h(\mathcal{S}),e(\mathcal{S}))\), where
\begin{itemize}
    \item \(h(\mathcal{S})\) is the minimum size of a hitting set for \(\mathcal{S}\), and
    \item \(e(\mathcal{S})\) is the smallest size of an egg-cut of \(\mathcal{S}\).
\end{itemize}
An example of a scramble on the cube graph \(Q_3\) is illustrated in Figure \ref{fig:Cube Scramble}, where \(\mathcal{S}=\{E_1,E_2,E_3,E_4\}\) and each \(E_i\) consists of two circled vertices.  Since the four eggs are disjoint, we have \(h(\mathcal{S})=4\).  We claim that \(e(\mathcal{S})=4\).  Certainly there exists an egg-cut of size \(4\), such as choosing \(A\) to be a single egg.  Indeed, for any set \(A\) with \(E_i\subset A\) and \(E_j\subset A^C\), then there must be at least two edges connecting \(A\) and \(A^C\) on both the inner and outer four-cycle of the graph.  Thus \(||\mathcal{S}||=4\).

\begin{figure}[hbt]
    \centering
    \includegraphics{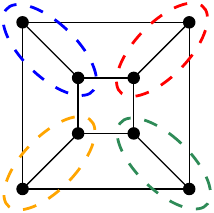}
    \caption{A scramble of order four on the cube graph}
    \label{fig:Cube Scramble}
\end{figure}

The \emph{scramble number} of a graph \(G\) is the maximum order of any scramble on \(G\). Thus the cube graph \(Q_3\) has \(\textrm{sn}(Q_3)\geq 4\); we will see later that in fact \(\textrm{sn}(Q_3)= 4\), so there exists no scramble of higher order on \(Q_3\).  By \cite[Theorem 1.1]{harp2020new}, the scramble number of a graph is at least as large as its treewidth:
\[\tw(G)\leq \sn(G).\]

Originally, scrambles and scramble numbers were defined on connected graphs.  In this paper we allow possibly disconnected graphs; fortunately, we have the following result to help us understand how scramble numbers of graphs related to those of connected components.  In particular, it turns out an optimal scramble must have all its eggs in one basket.

\begin{lemma}\label{lemma:eggs_in_one_basket}
If \(G\) is a graph with connected components \(G_1,\ldots,G_c\), then
\[\textrm{sn}(G)=\max_i\{\textrm{sn}(G_i)\}\]
\end{lemma}

\begin{proof}
Let \(\mathcal{S}\) be a scramble, and suppose that it has eggs in two different components, say \(E_i\subset G_i\) and \(E_j\subset G_j\).  Then \(A=G_i\) is an egg-cut for \(A\), but \(|E(G_i,G_i^C)|=0\), so \(||\mathcal{S}||=0\).

Thus for \(\mathcal{S}\) to be a scramble of positive order, all its eggs must be contained in a single graph \(G_i\).  This can be viewed as a scramble on \(G_i\), and in fact it has the same order as the corresponding scramble on \(G_i\):  certainly the minimal hitting set in the same; and minimal egg-cuts must all be of the form \(A\) where \(A\) is a proper nonempty subset of \(G_i\), and no edges outside of \(G_i\) contribute to a minimal egg-cut.  Thus the best we can do for \(G\) is to construct the largest order possible scramble on one connected component of \(G\), giving the claimed equality.
\end{proof}

We can similarly determine scramble number based on subgraphs in the event that a graph is only \(1\)-edge-connected. To do this, we need the following result.

\begin{proposition}[{\cite[Proposition 4.3]{harp2020new}}]\label{prop:subgraph}  If \(G\) is a subgraph of \(H\), then \(\textrm{sn}(G)\leq \textrm{sn}(H)\).  
\end{proposition}

\begin{lemma}\label{lemma_delete_bridge}
Let \(G\) be a graph with \(\lambda(G)=1\), and let \(e=uv\) be an edge such that deleting \(e\) disconnects \(G\) into \(G_1\) and \(G_2\). Then $\sn(G)=\max(\sn(G_1),\sn(G_2))$
\end{lemma}

\begin{proof}
First we have \(\sn(G)\geq \max(\sn(G_1),\sn(G_2))\) by Proposition \ref{prop:subgraph}, since \(G_1\) and \(G_2\) are both subgraphs of \(G\).

Now let \(\mathcal{S}\) be a scramble on \(G\); we will construct a scramble with at least as large of an order on either \(G_1\) or \(G_2\).  First assume there exist \(E_1,E_2\in \mathcal{S}\) with \(E_1\subset V(G_1)\) and \(E_2\subset V(G_2)\). Then taking \(A=V(G_1)\) gives an egg-cut of size \(1\), yielding \(||\mathcal{S}||=1\).  Since any scramble on \(G_1\) or \(G_2\) has scramble order at least \(1\), we then have at least as large of an order on either \(G_1\) or \(G_2\).

Now assume there exist no such \(E_1\) and \(E_2\).  Without loss of generality, we may assume that every egg of \(\mathcal{S}\) is either a subset of \(V(G_1)\), or contains the vertex \(u\) of the cut-edge \(e=uv\) that is an element of \(V(G_1)\).  Based on the scramble \(\mathcal{S}=\{E_1,\ldots,E_\ell\}\) on \(G\), construct a scramble \(\mathcal{S}'=\{E_1',\ldots,E_\ell'\}\) on \(G_1\) by setting \(E_i'=E_i\cap V(G_1)\).  We claim that \(||\mathcal{S}'||\geq ||\mathcal{S}||\).  First, any hitting set for \(\mathcal{S}'\) is also a hitting set for \(\mathcal{S}\), so \(h(\mathcal{S}')\geq h(\mathcal{S})\).  Next, let \(A\subset V(G_1)\) be a minimal egg-cut for \(\mathcal{S}'\), say with \(E_i'\subset A\) and \(E_j'\subset A^C\). Without loss of generality, assume that \(u\in A^C\).  Then \(A\) will form an egg-cut in \(G\) for \(\mathcal{S}\) with the same number of edges connecting \(A\) and \(A^C\).  Thus we have \(e(\mathcal{S}')\geq e(\mathcal{S})\).  Taken together, these inequalities give us \(||\mathcal{S}'||\geq ||\mathcal{S}||\).

Thus given a scramble on \(G\) we can construct a scramble either on \(G_1\) or \(G_2\) of at least the same order.  This means  \(\sn(G)\leq \max(\sn(G_1),\sn(G_2))\), completing the proof.
\end{proof}

Sadly a similar result does not hold for graphs with \(\kappa(G)=1\), as demonstrated in the following example.
\begin{example}\label{example:wedge_sum_scramble} Given two graphs \(G\) and \(H\) with vertices \(u\in V(G)\) and \(v\in V(H)\), the \emph{wedge sum} of \(G\) and \(H\) at \(u\) and \(v\) is obtained by gluing \(G\) and \(H\) together at the vertices \(u\) and \(v\).  
Let \(G\) be the slashed diamond graph pictured on the left of Figure \ref{figure:wedge_sum}.  The middle and right graphs in the same figure are two non-isomorphic ways of taking the wedge sum of \(G\) with itself.

\begin{figure}[hbt]
    \centering
    \includegraphics{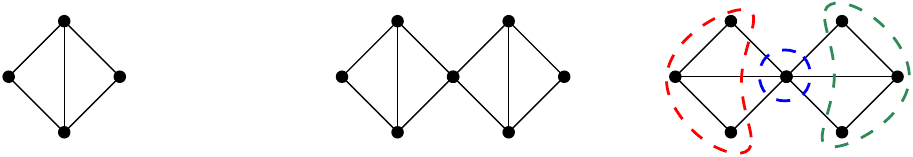}
    \caption{A graph and two ways of taking its wedge sum with itself, yielding different scramble numbers}
    \label{figure:wedge_sum}
\end{figure}

As none of these connected graphs are trees, they have scramble number at least \(2\) by \cite[Corollary 4.7]{harp2020new}; and the rightmost graph has scramble number at least \(3\), as demonstrated by the scramble of order \(3\) illustrated on it.  It turns out that these three graphs have scramble numbers equal to \(2\), \(2\), and \(3\) (from left to right); we will prove this in Example \ref{example:wedge_sum_gonality}.  This demonstrates that the scramble number of a wedge sum cannot be determined from the scramble numbers of the summands.
\end{example}

The next result is a useful lower bound on scramble number.

\begin{lemma}\label{lemma:edge-connectivity}
We have $\textrm{sn}(G)\geq\min(\lambda(G),|V(G)|)$, where $\lambda$ denotes edge-connectivity.
\end{lemma}

\begin{proof}
Consider the scramble on $G$ where each vertex is an egg. The minimum size of a hitting set is $|V(G)|$, and the smallest possible number of edges between any nonempty $A$ and $A^C$ (which will always separate eggs) is $\lambda(G)$.    Thus the order of that scramble is $\min(\lambda(G),|V(G)|)$.
\end{proof}

We close this subsection with a discussion of how scramble number behaves under various graph operations.

\begin{proposition}[Proposition 4.4 in \cite{harp2020new}]
 If \(G\) is a subdivision of \(H\) (that is, can be obtained from \(H\) by adding in \(2\)-valent vertices in the middle of edges of \(H\)), then \(\textrm{sn}(G)= \textrm{sn}(H)\).
\end{proposition}

Scramble number is not well-behaved under other graph relations.  As shown in \cite[Example 4.2]{harp2020new}, it is possible for \(G\) to be a minor of \(H\) with \(\sn(G)>\sn(H)\).  Originally, the authors of the current paper believed that scramble number would behave more nicely under the immersion minor relation; as shown in the following example suggested to the authors by Benjamin Baily, this is unfortunately not the case.

\begin{example}
Consider the graphs \(G\) and \(H\) illustrated in Figure \ref{figure:bens_counterexample}; note that \(G\) is an immersion minor of \(H\) by performing a lift along \(a,c,d\) and then another lift along \(a,d,e\) (these collections of edges are highlighted).  We claim that \(\sn(G)>\sn(H)\), meaning that scramble number is not immersion minor monotone. We first argue that \(G\) has scramble number at least \(3\).  Letting \(\mathcal{S}=\{\{a,e\},\{b,c\},\{d,f\}\}\), we certainly have \(h(\mathcal{S})=3\) since the eggs are disjoint; and \(e(\mathcal{S})=3\) since no \(2\)-edge-cut of \(G\) separates two eggs.  Thus \(\sn(G)\geq 3\).

\begin{figure}
    \centering
    \includegraphics{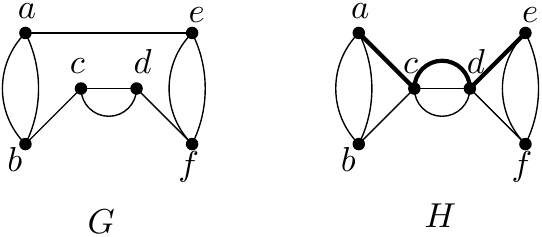}
    \caption{A graph \(G\) of scramble number 3, which is an immersion minor of \(H\), a graph of scramble number 2}
    \label{figure:bens_counterexample}
\end{figure}

We now argue that \(\sn(H)<3\).   First we note that if we delete one of the edges connecting \(c\) and \(d\), the resulting graph \(H'\)  has gonality \(2\) (defined in the next subsection), since the divisor \((a)+(b)\) would have rank \(1\).  It follows that \(H'\) has scramble number \(2\).  Suppose for the sake of contradiction that \(H\) has a scramble of order at least \(3\).  This scramble can be restricted to \(H'\), and must have a lower order; the size of a minimal hitting set cannot drop, so the size of a minimal egg cut must drop instead. The only way an egg cut could decrease in size is if \(c\) is on one side and \(d\) is on the other; it follows that \(A=\{a,b,c\}\) is an egg cut of size \(3\) for the scramble \(\mathcal{S}\) on \(H\), so there exist eggs \(E_1\subset A\) and \(E_2\subset A^C\).  If there exists a third egg \(E_3\subset \{a,b\}\), then \(A=\{a,b\}\) would form an egg cut of size \(2\) separating \(E_3\) from \(E_2\), which is impossible; a similar contradiction occurs if an egg were contained in \(\{e,f\}\).  Thus every egg must contain \(c\) or \(d\); but then \(\{c,d\}\) is a hitting set of size \(2\), a contradiction.  We conclude that \(\textrm{sn}(H)=2\). 
\end{example}

\subsection{Divisors, chip-firing, and gonality on finite graphs}

Let \(G=(V,E)\) be a connected graph, with multiple edges allowed but not loops.

We define a \emph{divisor} \(D\) on \(G\) to be an element of the free abelian group generated by the vertices of \(G\):
\[D=\sum_{v\in V}a_v (v),\,\,\,\,\, a_v\in\mathbb{Z}.\]
We intuitively think of \(D\) as a placement of (integer numbers of) chips on the vertices of \(G\).  The \emph{degree} of a divisor is the sum of the coefficients.  We say that \(D\) is \emph{effective}, written \(D\geq 0\), if every \(a_v\geq 0\). Whenever \(a_v<0\), we refer to \(v\) being ``in debt''.

We can transform one divisor into another by a \emph{chip-firing move}.  For a vertex \(w\), the divisor obtained from \(D=\sum_{v\in V}a_v(v)\) by \emph{chip-firing \(w\)} is
\[D'=(a_w-\textrm{val}(w))(w)+\sum_{v\neq w}(a_v+|E(\{v\},\{w\})|)(v).\]
Thus the vertex \(w\) loses
\(\textrm{val}(w)\) chips, and each neighbor \(v\) of \(w\) gains a number of chips equal to the number of edges connecting \(v\) and \(w\).  We say two divisors \(D_1\) and \(D_2\) are \emph{equivalent}, written \(D_1\sim D_2\), if \(D_2\) can be obtained from \(D_1\) by a sequence of chip-firing moves.

\begin{figure}[hbt]
    \centering
\includegraphics[scale=0.8]{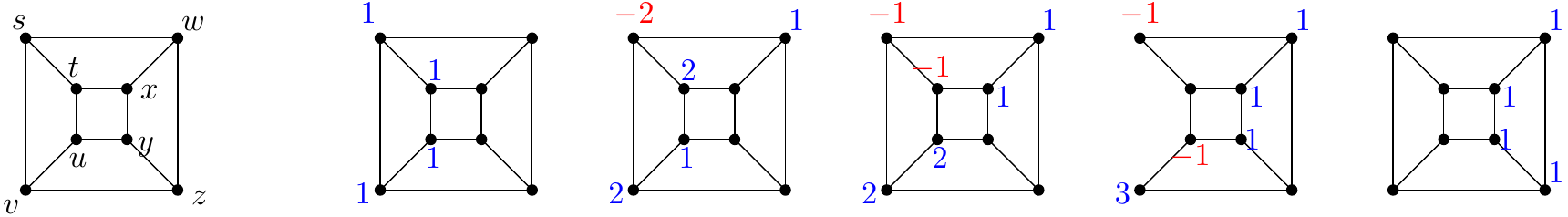}
    \caption{The cube graph with vertices labelled, and five divisors on it that are equivalent to one another}
    \label{figure:cube_chip_firing}
\end{figure}

Figure \ref{figure:cube_chip_firing} illustrates a labelling on the vertices of the cube \(Q_3\), followed by five divisors on the cube \(Q_3\) that are equivalent to one another. We can move from the first divisor \((s)+(t)+(u)+(v)\) to the second divisor \(-2(s)+2(t)+(u)+(v)+(w)\) by chip-firing \(s\), then to the third divisor by chip-firing \(t\), to the fourth by chip-firing \(u\), and finally chip-firing \(v\) to obtain the divisor \((w)+(x)+(y)+(z)\).  Of these five divisors,  only the first and the last are effective.  Sometimes it is convenient to speak of ``simultaneously'' firing a set of vertices, meaning that we fire those vertices in any order; for instance, we can move from the first divisor to the fifth by simultaneously firing the set of vertices \(\{s,t,u,v\}\).  We remark that this notion of simultaneous chip-firing makes it more clear why \(\sim\) is an equivalence relation:  to undo chip-firing a vertex \(q\), we can simultaneously fire the set of vertices \(V(G)\setminus\{q\}\).

We define the \emph{rank} \(r(D)\) of a divisor \(D\) as \(-1\) if \(D\) is not equivalent to an effective divisor; and as \(r\) otherwise, where \(r\) is the maximum integer such that for every effective divisor \(E\) of degree \(r\), the divisor \(D-E\) is equivalent to an effective divisor.  Intuitively, \(r(D)\) is the largest amount of debt that, no matter where it is added to the graph, can be eliminated by \(D\) through chip-firing moves.  The \emph{gonality} of a graph \(G\), denoted \(\gon(G)\) is the minimum degree of a positive rank divisor on \(G\).

Graph gonality is a graph-theoretic version of the gonality of an algebraic curve, which can either be defined in terms of divisor theory on curves, or as the minimum degree of a nondegenerate morphism to a projective line \cite[Section 8C]{algebraic_gonality}.  This leads us to borrow further pieces of terminology from the algebro-geometric world.  For instance, if a graph on at least three vertices has gonality \(2\), we call it \emph{hyperelliptic} as in \cite{bn2}.

We claim that the divisors pictured in Figure \ref{figure:cube_chip_firing} each have positive rank.  As they are equivalent divisors, it suffices to argue this for only one of them.  Consider the leftmost divisor \(D\), and let \(E\) be an effective divisor of degree \(1\).  If \(E=(s),(t),(u),\) or \((v)\), then \(D-E\) is already effective.  If \(E\) is any other divisor, we may perform the series of chip-firing movies pictured in Figure \(3\) to move chips onto the other four vertices, thereby eliminating debt.  Thus \(D-E\) is always equivalent to an effective divisor.  This means \(r(D)\geq 1\), implying that \(\gon(Q_3)\leq \deg(D)=4\).

For any connected graph \(G\) we have
\[\tw(G)\leq \sn(G)\leq \gon(G)\]
by \cite[Theorem 1.1]{harp2020new}. For the running example of the cube graph, we now have \(4\leq\sn(Q_3)\leq \gon(Q_3)\leq 4\), implying that \(\sn(Q_3)=\gon(Q_3)=4\). For \(G\)  a simple graph, we have that edge-connectivity is bounded by treewidth, and so we have  \[\kappa(G)\leq \lambda(G)\leq\tw(G)\leq \sn(G)\leq \gon(G).\]

\begin{example}\label{example:wedge_sum_gonality} Consider the divisors illustrated on the three graphs in Figure \ref{fig:wedge_sum_chips}; these are the same graphs from Example \ref{example:wedge_sum_scramble}.  We claim that these divisors have positive rank.  For the first graph, firing the leftmost vertex moves two chips to the middle two vertices; and firing all but the rightmost vertex then moves the two chips to that vertex.  Thus wherever a \(-1\) is placed, chips can be moved to eliminate the debt.  A similar argument holds for the second graph.  For the third graph, firing the leftmost four vertices moves chips to the rightmost three, and similarly with left and right switched, so debt can be eliminated anywhere.  Thus the graphs have gonality at most \(2\), \(2\), and \(3\), respectively.  Since their scramble numbers were at least \(2\), \(2\), and \(3\), and since scramble number is a lower bound on gonality, we have that each graph has scramble number equal to gonality (equal to \(2\) for the first two graphs, and equal to \(3\) for the rightmost graph).

\begin{figure}[hbt]
    \centering
    \includegraphics{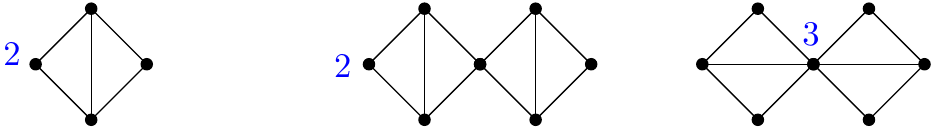}
    \caption{Divisors of positive rank on three graphs}
    \label{fig:wedge_sum_chips}
\end{figure}

\end{example}

We now present two examples that summarize many families of graphs for which gonality is known.

\begin{example}[Graphs with \(\lambda(G)=\gon(G)\)]  Any tree \(T\) has \(\lambda(T)=\gon(T)=1\) (indeed, a graph has gonality \(1\) if and only if it is a tree \cite[Lemma 1.1]{bn2}).  The complete graph \(K_n\) on \(n\) vertices has \(\lambda(K_n)=\gon(K_n)=n-1\); and more generally a complete multipartite graph \(K_{n_1,\ldots,n_\ell}\) has  \(\lambda(G)=\gon(K_{n_1,\ldots,n_\ell})=n_1+\cdots+n_\ell-\max_i\{n_i\}\) \cite[Example 4.3]{treewidth}.  Any cycle graph \(C_m\) has \(\lambda(C_m)=\gon(C_m)=2\); the same holds for any bridgeless hyperelliptic graph on more than \(2\) vertices.  If \(G\) is a trivalent, bridgeless, simple graph of gonality \(3\), then \(\lambda(G)=\gon(G)=3\) \cite{gonality_three}.

\label{example:lambda=gon}
\end{example}

\begin{example}[Product graphs with previously known gonality]  A two-dimensional grid graph \(G_{m,n}\) is the product \(P_m\square P_n\) of two path graphs, and has gonality \(\min(m,n)\) \cite{treewidth}; more generally the product \(T_1\square T_2\) of any two trees has gonality \(\min(|V(T_1)|,|V(T_2)|)\) \cite[Proposition 11]{gonality_product}.  The product \(P_m\square C_n\) of a path with a cycle has gonality \(\min(2m,n)\), and the product \(C_m\square C_n\) of a cycle with a cycle has gonality \(\min(2m,2n)\); this was proved for most values of \(m\) and \(n\) using treewidth in \cite{glued_grid_graphs}, and in general using scramble number in \cite{harp2020new}. The product \(T\square K_n\) of a tree \(T\) on at least two vertices with a complete graph \(K_n\) has gonality \(n\) \cite{gonality_product}.  Finally, for \(m\leq \min(n,5)\), the \(m\times n\) rook's graph \(K_m\square K_n\) has gonality \((m-1)n\) \cite{gonality_product}.

\label{example:known_products}
\end{example}

A useful tool for studying gonality is \emph{Dhar's burning algorithm} \cite{dhar}.  Suppose \(D\) is an effective divisor, and we want to check if \(D-(q)\) is equivalent to an effective divisor.  If \(D-(q)\) has no vertex in debt, then we are done.  Otherwise, start a ``fire'' at \(q\), which spreads through \(G\) as follows:  if a vertex is on fire, then every edge incident to it burns.  If a vertex has more burning edges incident to it than it has chips from \(D\), then that vertex burns.  Let the fire propagate through \(G\).  If all of \(G\) burns, then debt cannot be removed.  If not all of \(G\) burns, then chip-fire the unburned vertices; this will not introduce any new debt in the graph.  If \(q\) is out of debt, then we are done; if not, then run the burning process again, repeating until either \(q\) is out of debt or the whole graph burns.  If there is a way to remove debt, this process will find it; and if it is impossible to do so, then the whole graph will eventually burn.  A useful consequence of this algorithm is the following lemma.

\begin{lemma}\label{lemma:dhars}  Suppose \(D\) is an effective divisor on a graph \(G\), and that there exists a vertex \(q\) such that the burning process on \(D-(q)\) causes the whole graph to burn.  Then \(r(D)=0\).
\end{lemma}

Some graphs are particularly susceptible to arguments involving Dhar's burning algorithm, such as the complete graph on \(n\) vertices \(K_n\).  We recall the following lemma.

\begin{lemma}[{\cite[Lemma 14]{gonality_product}}]  \label{lemma:complete_dhars}   (\cite[Lemma 14]{gonality_product}) Let \(D\) be an effective divisor of degree at most \(n-2\) on \(K_n\), and let \(q\) be a vertex with no chips from \(D\). Running the burning process on \(D-(q)\) causes the whole graph to burn in one iteration of Dhar's algorithm.
\end{lemma}

It immediately follows that \(\gon(K_n)\geq n-1\):  any effective divisor of smaller degree must have rank \(0\) by Lemma \ref{lemma:dhars}.

Another useful set of tools when studying gonality comes from \cite{spencer}, which associates divisors to partial orientations on the graph.  We state the following consequence of that work, an immediate consequence of  \cite[Theorem 1.3]{spencer}.

\begin{lemma}\label{lemma:spencer}  Let \(D\) be a divisor with \(r(D)\geq 0\) and \(\deg(D)\leq |E(G)|-|V(G)|\).  There exists an equivalent effective divisor \(D'\) such that \(D'\) has at most \(\val(v)-1\) chips on each vertex \(v\).
\end{lemma}

\subsection{Metric graphs}  In this subsection we briefly recall divisor theory on metric graphs. Given a finite connected graph \(G\), we can assign lengths to its edges via a \emph{length function} \(\ell:E(G)\rightarrow\mathbb{R}_{>0}\).  The pair \((G,\ell)\) then defines a topological space \(\Gamma\), constructed by taking a line segment of length \(\ell(e)\) for each \(e\in E(G)\) and gluing the line segments together at their endpoints according to the structure of \(G\). We refer to \(\Gamma\) as a \emph{metric graph}, and \((G,\ell)\) a model for \(\Gamma\). Note that distinct pairs \((G,\ell)\) and \((G',\ell')\) may yield the same metric graph.  Often we work with the \emph{canonical loopless model} \((G,\ell)\) of \(\Gamma\), such that \(G\) has no loops and such that no vertex of \(G\) has degree \(2\) unless it is incident to the same vertex twice.

The theory of divisors on finite graphs was extended to metric graphs in \cite{gk} and \cite{mz}.  A divisor \(D\) on \(\Gamma\) is now a finite formal sum of points of \(\Gamma\) with integer coefficients; the main difference with the case of finite graphs is that these points are now allowed to be on the interior of edges.  Many pieces of terminology, such as \emph{degree} and \emph{effective}, carry through immediately.  Equivalence of divisors is defined in terms of tropical rational functions.  A \emph{tropical rational function} is a continuous piece-wise linear function \(f:\Gamma\rightarrow\mathbb{R}\) with finitely many pieces and integer slopes.  We define the order of \(f\) a point \(p\) to be the sum of the outgoing slopes of \(f\) at the point \(p\), and denote this integer as \(\textrm{ord}_f(p)\).  Note that \(\textrm{ord}_f(p)=0\) for all but finitely many choices of \(p\); this allows us to define the divisor
\[\textrm{div}(f)=\sum_{p\in\Gamma}\textrm{ord}_f(p)\cdot (p).\]
We then say that \(D\) and \(E\) are equivalent if \(D-E=\textrm{div}(f)\) for some \(f\).

Now that we have a definition of equivalence, we say that the rank \(r(D)\) of a divisor \(D\) is \(-1\) if \(D\) is not equivalent to any effective divisor; and that otherwise \(r(D)\) is the maximum integer \(r\geq 0\) such that for every effective divisor \(E\) of degree \(r\), \(D-E\) is equivalent to an effective divisor.  The gonality of \(\Gamma\), denoted \(\textrm{gon}(\Gamma)\), is then the minimum degree of a positive rank divisor on \(\Gamma\).  A priori, it might seem intractable to show a divisor has positive rank, since there are infinitely many effective divisors \(E=(p)\) of degree \(1\).  Thanks to work in \cite{rank_determining}, however, it suffices to consider \(E\) with support in what is called a \emph{rank determining set}. 

We present two lemmas on the gonality of metric graphs, which follow quickly from previous results and arguments in the literature.

\begin{lemma}\label{lemma:metric_lower_bound}  Let \(\Gamma\) be a metric graph with canonical loopless model \((G,\ell)\).  We have
\[\textrm{sn}(G)\leq \gon(\Gamma).\]
\end{lemma}

\begin{proof}  By \cite[Theorem 5.1]{treewidth}, there is a subdivision \(H\) of \(G\) such that \( \textrm{gon}(H)\leq \textrm{gon}(\Gamma)\).  By \cite[Proposition 4.4]{harp2020new} we know that \(\textrm{sn}(G)=\textrm{sn}(H)\), so we have
\[\textrm{sn}(G)=\textrm{sn}(H)\leq \textrm{gon}(H)\leq \textrm{gon}(\Gamma),\]
as desired.
\end{proof}

\begin{lemma}\label{lemma:metric_upper_bound}
 Let \(\Gamma\) be a metric graph with canonical loopless model \((G,\ell)\), and suppose that \(G\) is simple with \(|V(G)|=n\).  Then we have
\[\gon(\Gamma)\leq n-\alpha(G).\]
\end{lemma}

This result was proved in \cite[Theorem 3.1]{gonality_random_graphs} for finite graphs; our proof follows theirs closely, with minor modifications for our metric case.

\begin{proof}
Let \(S\subset V(G)\) be an independent set of maximum size, and let \(D\) place a chip on all points of \(\Gamma\) corresponding to \(S^C\subset V(G)\).  Then \(D\) has degree \(n-\alpha(G)\).  To show that it has positive rank, we let \((p)\) be any effective divisor of degree \(1\) such that \(p\in V(G)\); by \cite[Theorem 1.6]{rank_determining}, \(V(G)\) is a rank determining set, so it suffices to show that \(D-(p)\) is equivalent to an effective divisor.  If \(p\in S^C\), then \(D-(p)\) is effective.  If \(p\in S\), we let \(e_1,\ldots,e_k\) be the edges incident to \(p\), and without loss of generality we assume \(e_1\) has the minimum length of those \(k\) edges.  Let \(f\) be the tropical rational function that is \(0\) outside of  \(e_1,\ldots,e_k\); has slope \(1\) for a length of \(\ell(e_1)\) on each of \(e_1,\ldots,e_k\) heading in towards \(p\); and has slope \(0\) otherwise.  Note that the only points \(q\) such that \(\textrm{ord}_f(q)>0\) are the endpoints of \(e_1,\ldots,e_k\) besides \(p\); and that \(\textrm{ord}_f(p)<0\).  Thus \(D-(p)-\textrm{div}(f)\) is an effective divisor that is equivalent to \(D-(p)\).  We conclude that \(r(D)\geq 1\), completing the proof.
\end{proof}

\section{The edge scramble and computational complexity}
\label{section:complexity}

Given a graph \(G\), we define the \emph{edge scramble on \(G\)}, denoted \(\mathcal{E}(G)\), to be the scramble on \(G\) whose eggs consist of all edges of \(G\).  A set \(S\subset V(G)\) is hitting set for \(\mathcal{E}(G)\) if and only if \(S\) is a vertex cover, meaning every vertex is either in \(S\) or adjacent to a vertex in \(S\); thus the minimum size of a hitting set for \(\mathcal{E}\) is \(n-\alpha(G)\), where \(n=|V(G)|\) and  \(\alpha(G)\) is the independence number of \(G\).  The challenge in computing the order of \(\mathcal{E}(G)\) is then in finding the size of a minimal egg-cut.  In the event that the minimum degree \(\delta(G)\) of \(G\) is large compared to the number of vertices, we can achieve this using the following lemma.

\begin{lemma}
Let \(G\) be a simple graph on \(n\) vertices with \(\delta(G)\geq n/2+1\).  Then the order of the edge scramble \(\mathcal{E}(G)\) is \(n-\alpha(G)\).
\end{lemma}

\begin{proof}
First we note that \(G\) is connected:  the maximum possible value of \(\delta(G)\) for a disconnected graph is \(\lfloor n/2\rfloor-1\).
Since the size of a minimal hitting set of \(\mathcal{E}(G)\) is \(n-\alpha(G)\), it suffices to show that any egg-cut has size at least \(n-\alpha(G)\).  Indeed, we will show that any egg-cut has size at least \(n\).

Let \(A\subset V(G)\) such that \((A,A^C)\) forms an egg-cut.  Since every egg has two vertices, we know \(|A|\geq 2 \); and without loss of generality \(|A|\leq |A^C|\), so \(|A|\leq n/2\).

To lower bound \(|E(A,A^C)|\), we note that every vertex in \(A\) has at least \(\delta\) edges incident to it, giving us at least \(\delta|A|\) edges incident to a vertex in \(A\) (possibly counting some edges twice).  However, not all of these need be in \(E(A,A^C)\):  some might be connected pairs of vertices in \(A\), with the worst case being that every two vertices of \(A\) are incident to one another.  This implies that
\[|E(A,A^C)|\geq \delta\cdot |A|-2{|A|\choose 2}=\delta\cdot |A|-|A|(|A|-1),\]
where the factor of \(2\) in front of \({|A|\choose 2}\) comes from the fact that \(\delta\cdot |A|\) double counts any edges shared between vertices of \(A\).

We claim that for \(2\leq |A|\leq \lfloor n/2\rfloor \), we have \(\delta\cdot |A|-|A|(|A|-1)\geq n-1\).  We will show this by arguing that \(f(x)=\delta x-x(x-1)\geq n-1\) for all real numbers \(2\leq x\leq \lfloor n/2\rfloor\).  Since \(f(x)=-x^2+(\delta+1)x\) is concave down, its minimum value on an interval is achieved at an endpoint.  Note that \[f(2)=-4+2(\delta+1)=2\delta-2\geq 2(\lfloor n/2\rfloor +1)-2\geq n-1,\] and that \[f(\lfloor n/2\rfloor)=-\lfloor n/2\rfloor^2+(\delta+1)\lfloor n/2\rfloor\geq-\lfloor n/2\rfloor^2+(\lfloor n/2\rfloor+2)\lfloor n/2\rfloor=2\lfloor n/2\rfloor\geq n-1.\]  Thus there are at least \(n-1\) edges in \(E(A,A^C)\).

Since the minimum hitting set for \(\mathcal{E}(G)\) is \(n-\alpha(G)\), and since the size of a minimum egg-cut is at least \(n-1\geq n-\alpha(G)\), we conclude that the order of \(\mathcal{E}\) is \(n-\alpha(G)\).
\end{proof}

\begin{corollary}\label{corollary:all_equal}
Let \(G\) be a simple graph on \(n\) vertices with \(\delta(G)\geq \lfloor n/2\rfloor +1\).
Then
\[\textrm{sn}(G)=\textrm{gon}(G)=n-\alpha(G).\]
\end{corollary}

\begin{proof}
By the previous lemma the edge scramble on \(G\) has order \(n-\alpha(G)\), implying that  \(n-\alpha(G)\leq  \textrm{sn}(G)\).  On the other hand, \(\textrm{gon}(G)\leq n-\alpha(G)\) since an effective divisor with a single chip on every vertex outside of a maximal independent set has positive rank (see \cite[Theorem 3.1]{gonality_random_graphs}).  Thus we have
\[n-\alpha(G)\leq\textrm{sn}(G)\leq\textrm{gon}(G)\leq n-\alpha(G), \]
implying equality of all terms.
\end{proof}

\begin{example}
It is reasonable to ask whether this result is sharp, or if we could decrease the lower bound on \(\delta(G)\), perhaps with the added assumption that \(G\) is connected. For \(n\geq 6\) and even, it is indeed sharp:  consider \(K_m\square K_2\), the \(m\times 2\) rook's graph, where \(m\geq 3\).  This has \(n=2m\) vertices, each of which has degree \(m=n/2\), so \(\delta(K_m\square K_2)=n/2\).  The gonality of this graph is \(\min(m,2(m-1))=m\) by \cite{gonality_product}, but \(\alpha(K_m\square K_2)=2\), so \(\gon(K_m\square K_2)=m<2m-2=n-\alpha(K_m\square K_2)\) since \(m\geq 3\).

For \(n\) odd with \(n\geq 5\), we can construct a similar example to show that \(\delta(G)\geq \lfloor n/2\rfloor\) is not a sufficient hypothesis.  In particular, take the graph \(K_m\square K_2\) where \(m\geq 2\), and add a vertex \(v\) that is incident to all vertices in one of the two canonical copies of \(K_m\). This graph \(G\) has \(n=2m+1\) vertices, and so \(\delta(G)=m=\lfloor(2m+1)/2\rfloor= \lfloor n/2\rfloor\).  A treewidth argument shows that  \(\gon(G)\geq m\), and indeed by placing a chip on all vertices incident to \(v\) we can find a positive rank divisor of degree \(m\).  However, \(\alpha(G)=2\), so \(\gon(G)=m<2m-1=(2m+1)-2=n-\alpha(G)\) since \(m\geq 2\).
\end{example}

This result for finite graphs can be used to prove a similar result for metric graphs.

\begin{corollary}\label{corollary:metric}
Let \(\Gamma\) be a metric graph with canonical loopless model \((G,\ell)\), such that \(G\) is simple with \(|V(G)|=n\) and \(\delta(G)\geq \lfloor n/2\rfloor+1\).  Then \[\textrm{gon}(\Gamma)=\textrm{gon}(G)=\sn(G)=n-\alpha(G).\]
In particular, changing the edge lengths of such a metric graph \(\Gamma\) does not effect its gonality.
\end{corollary}

\begin{proof}
By Lemmas \ref{lemma:metric_lower_bound} and \ref{lemma:metric_upper_bound} we have
\[\textrm{sn}(G)\leq \textrm{gon}(\Gamma)\leq n-\alpha(G).\]
By Corollary \ref{corollary:all_equal}, the upper and lower bound are equal to one another, and to \(\textrm{gon}(G)\).  This gives the claimed equalities.
\end{proof}

We now present the construction that will allow our NP-hardness proofs.  These will make use of the \emph{\(\ell^{th}\) cone} over a graph \(G\); this graph is constructed from \(G\) by adding \(\ell\) additional vertices, each connected to every other vertex (including one another).

\begin{lemma}\label{lemma:ghat_construction}
Let \(G\) be a simple connected graph on \(m\geq 2\) vertices, and let \(\widehat{G}\) be the \(m^{th}\) cone over \(G\). Then \(\textrm{sn}(\widehat{G})=\textrm{gon}(G)=2m-\alpha(G)\).
\end{lemma}

\begin{proof}
First we note that \(\delta(\widehat{G})\geq m+1\):  every vertex in \(G\) started out with degree at least \(1\), and in \(\widehat{G}\) that degree is increased by \(m\); and every new vertex has degree \(2m-1\geq m+1\). Since \(\widehat{G}\) is  graph on \(n=2m\) vertices and \(\delta(\widehat{G})\geq m+1=\lfloor n/2\rfloor +1\) we can apply the previous corollary to determine that
\[\textrm{sn}(\widehat{G})=\textrm{gon}(\widehat{G})=2m-\alpha(\widehat{G}).\]
It remains to show that \(\alpha(\widehat{G})=\alpha(G)\). But taking the cone of a graph does not change the independence number, so the \(m^{th}\) cone over \(G\) has the same indepedence number as \(G\).  This completes the proof.
\end{proof}

We can now prove our NP-hardness results.

\begin{proof}[Proof of Theorems \ref{theorem:NP_hard_scramble}, \ref{theorem:NP_hard_gonality}, and \ref{theorem:NP_hard_metric}]  The graph \(\widehat{G}\) constructed from a simple graph \(G\) in Lemma \ref{lemma:ghat_construction} has a number of vertices and a number of edges that are computed by polynomial expressions of the corresponding numbers for \(G\) (in  particular: \(|V(\widehat{G})=2|V(G)|\) and \(|E(\widehat{G})|=|E(G)|+|V(G)|^2+\frac{|V(G)|(|V(G)|-1)}{2}\)).  Thus to compute the independence number \(\alpha(G)=2m-\textrm{sn}(\widehat{G})=2m-\textrm{gon}(\widehat{G})\), we can instead compute either the scramble number or the gonality of the (simple) graph \(\widehat{G}\) with only polynomial blow-up.  Similarly, letting \(\Gamma\) be any metric version of \(\hat{G}\), we have that \(\alpha(G)=2m-\gon(\Gamma)\) by Corollary \ref{corollary:metric}, so we can also compute independence number from metric graph gonality with only polynomial blow-up in the input.   Since computing independence number is NP-complete, we conclude that computing the scramble number of a graph is NP-hard;  that computing the gonality of a simple graph is NP-hard; and that computing the gonality of a metric graph is NP-hard.
\end{proof}

We close this section by posing the following question.

\begin{question}
For any fixed \(k\), does there exist a polynomial time algorithm for determining whether \(\textrm{sn}(G)\leq k\)?
\end{question}

The answer to the corresponding question for \(\gon(G)\) is yes; see for instance \cite[\S 6]{gonality_sequences}.  This comes from the brute-force algorithm of checking whether each divisor of degree \(k\) has positive rank, for instance via Dhar's algorithm.  However, there is not even a particularly nice ``brute-force'' algorithm for computing scramble number that we know of.

\section{A lower bound on scramble number for product graphs}
\label{section:lower_bound}

We now prove our main lower bound on the scramble number of the Cartesian product of two graphs. 

\begin{theorem}\label{theorem:sn_lower_bound_general}
Let \(G\) and \(H\) be connected graphs, such that \(G\) is \(k\)-connected and has at least \(2k-1\) vertices.  We have
\[\textrm{sn}(G\square H)\geq \min\left(k|V(H)|,|V(G)|\lambda(H),(|V(G)|-2k+2)\lambda(H)+2\lambda(G)\right). \]
\end{theorem}

\begin{proof}
Define a scramble   \(\mathcal{S}\) on \(G\square H\) whose eggs are  canonical copies of \(G\) with \(k-1\) vertices deleted.  To see that this is indeed a scramble, we note that \(G\) is \(k\)-connected, so deleting  \(k-1\) vertices will not disconnect a canonical copy of \(G\).

We claim that the size of a minimum hitting set of \(\mathcal{S}\) is \(k|V(H)|\).  Certainly if we choose \(k\) vertices in each of the \(|V(H)|\) canonical copies of \(G\), this forms a hitting set.  Indeed, a hitting set must include \(k\) vertices in each of the canonical copies of \(G\): otherwise one canonical copy of \(G\) would have at most \(k-1\) vertices in the hitting set, and we can construct an egg consisting of that canonical copy with all those hitting set vertices  deleted.  This establishes our claim.

We now need to bound the size of a minimum egg-cut in \(\mathcal{S}\).  Let \(E_1,E_2\in \mathcal{S}\) and \(A\subset V(G\square H)\) such that \(E_1\subset A\) and \(E_2\subset A^C\).  Since \(|V(G)|\geq 2k-1\), any two eggs in the same canonical copy of \(G\) must overlap, so we have that \(E_1\) and \(E_2\) are contained in different canonical copies of \(G\), say \(E_1\subset G\square w_1\) and \(E_2\subset G\square w_2\) where \(w_1\neq w_2\).

 By Theorem \ref{theorem:mengers}, we know that there are at least \(\lambda(H)\) edge-disjoint paths from \(v\square w_1\) and \(v\square w_2\) within \(v\square H\).  Thus if  \(v\square w_1\in A\) and \(v\square w_2\in A^C\), there must be at least \(\lambda(H)\) edges from the egg-cut in \(v\square H\).  Letting \(S\) denote the number of \(v\in V(G)\) that do not have both \(v\square w_1\in A\) and \(v\square w_2\in A^C\), we therefore have at least \((|V(G)|-|S|)\lambda(H)\) edges in the egg-cut
    
 We also know that if there exists a vertex \(v\in V(G)\) with \(v\square w_1\in A^C\), then there are at least \(\lambda(G)\) edges of the egg-cut in \(G\square w_1\); this is because  that number of edges are required to separate \(v\square w_1\) from \(E_1\) within that canonical copy of \(G\).  A similar result holds if there exists a vertex \(v\in V(G)\) with \(v\square w_2\in A\).  Let \(p\in \{0,1,2\}\) count how many of \(V(G\square w_1)\) and \(V(G\square w_2)\) are not completely contained in  \(A\) and \(A^C\), respectively.  We therefore have at least \(p\lambda(G)\) edges in our egg-cut within \(G\square w_1\) and  \(G\square w_2\).  Thus, at minimum, our egg-cut has
 \[(|V(G)|-|S|)\lambda(H)+p\lambda(G)\]
 edges.  To remove the dependence on \(|S|\), we split into the following three cases.
    
    \begin{itemize}
    \item[(i)] \(p=0\).  In this case, we have \(|S|=0\) since for all \(v\in V(G)\) we have \(v\square w_1\in A\) and \(v\square w_2\in A^C\).  This gives us at least \(|V(G)|\lambda(H)\) edges.
    \item[(ii)] \(p=1\).   Without loss of generality, say \(V(G\square v_2)\subset A^C\), and that \(V(G\square v_1)\not\subset A\).  In this case, the largest \(S\) could be is if it consists of all \(v\in V(G)\) such that \(v\square w_1\notin E_1\); since eggs take up all but \(k-1\) vertices of a canonical copy of \(G\), we have \(|S|\leq k-1\).  This means we have at least \(|V(G)-k+1)\lambda(H)+\lambda(G)\) edges.
    \item[(iii)] \(p=2\). Let \(T_1\subset V(G)\) be the set of all \(v\) such that \(v\square w_1\notin A\); and let \(T_2\subset V(G)\) be the set of all \(V\) such that \(v\square w_2\notin A^C\).  As argued the previous case, we have \(|T_i|\leq k-1\). Note that \(S=T_1\cup T_2\), meaning that \(|S|\leq |T_1|+|T_2|\leq 2k-2\).  This means we have at least  \((|V(G)|-2k+2)\lambda(H)+2\lambda(G)\) edges.
    \end{itemize}
  
The number of edges separating \(A\) and \(A^C\) is therefore at least
\[\min(|V(G)|\lambda(H),(|V(G)|-k+1)\lambda(H)+\lambda(G),(|V(G)|-2k+2)\lambda(H)+2\lambda(G)).\]
We can in fact omit the middle term:  if \((|V(G)|-k+1)\lambda(H)+\lambda(G)\leq |V(G)|\lambda(H)\), then
\[(-k+1)\lambda(H)+\lambda(G)\leq 0.\]
It follows then that
\begin{align*}(|V(G)|-k+1)\lambda(H)+\lambda(G)\,\geq\,& (|V(G)|-k+1)\lambda(H)+\lambda(G)+(-k+1)\lambda(H)+\lambda(G)\\\,=\,&(|V(G)|-2k+2)\lambda(H)+2\lambda(G)).
\end{align*}
Thus if the second term is smaller than the first term, then it is larger than the third term; so the second term is never the (unique) minimum.  The number of edges separating \(A\) and \(A^C\) is therefore at least
\[\min(|V(G)|\lambda(H),(|V(G)|-2k+2)\lambda(H)+2\lambda(G)).\]
Taking the minimum of this and the minimum size of our hitting set gives us the following lower bound on the order of the scramble \(\mathcal{S}\):
\[\min(k|V(H)|,|V(G)|\lambda(H),(|V(G)|-2k+2)\lambda(H)+2\lambda(G)).\]
This therefore serves as a lower bound on \(\textrm{sn}(G)\) as claimed.
\end{proof}

We immediately deduce the following result.
\begin{corollary}\label{corollary:k=1}
If \(G\) and \(H\) are connected graphs on at least \(2\) vertices, then
\[\textrm{sn}(G\square H)\geq \max(\min(|V(H)|,|V(G)|\lambda(H)),\min(|V(G)|,|V(H)|\lambda(G))\]
\end{corollary}

\begin{proof}
By the symmetry of assumptions on \(G\) and \(H\), it's enough to show
\[\textrm{sn}(G\square H)\geq \min(|V(H)|,|V(G)|\lambda(H)).\]
Since \(G\) is \(1\)-connected, we may apply Theorem \ref{theorem:sn_lower_bound_general} with \(k=1\) to find
\[\textrm{sn}(G\square H)\geq \min(|V(H)|,|V(G)|\lambda(H),|V(G)|\lambda(H)+2\lambda(G)).\]
The third term is always larger than the second, so we have \[\textrm{sn}(G\square H)\geq \min(|V(H)|,|V(G)|\lambda(H)),\]
as desired.
\end{proof}

We can slightly improve the bound from Theorem \ref{theorem:sn_lower_bound_general} in the case of \(k=2\).

\begin{proposition}
\label{proposition:k=2} 
If \(G\) and \(H\) are connected graphs with \(\kappa(G)\geq 2\), then
\[\textrm{sn}(G\square H)\geq \min\left(2|V(H)|,|V(G)|\lambda(H),(|V(G)|-2)\lambda(H)+2\delta(G)\right). \]
\end{proposition}

\begin{proof}
The proof is nearly identical to that of Theorem \ref{theorem:sn_lower_bound_general} with \(k=2\); note that \(G\) must have least \(2\cdot 2-1=3\) vertices by the \(2\)-connectivity assumption.  The one difference is when there are vertices in \(G\square w_1\) not in \(A\) (or similarly when there are vertices in \(G\square w_2\) not in \(A^C\)).  Since \(k-1=1\), there must be exactly one such vertex in \(G\square w_1\) (or in \(G\square w_2\)), meaning that there are at least \(\delta(G)\) (rather than just \(\lambda(G)\)) edges separating \(A\) from \(A^C\) in \(G\square w_1\) (or in \(G\square w_2\)).  This gives us the claimed formula.
\end{proof}

\section{Applications to gonality}
\label{section:applications_to_gonality}

We know that the scramble number of a graph is a lower bound on the gonality of that graph.  We also know by \cite{gonality_product} that 
\[\textrm{gon}(G\square H)\leq |V(H)|\textrm{gon}(G)\]
and
\[\textrm{gon}(G\square H)\leq |V(G)|\textrm{gon}(H).\]
In this section we show that these upper bounds match our lower bounds on scramble number for many families of products graphs, allowing us to determine their gonality.  We separate our graphs by the choice of \(k\) from Theorem \ref{theorem:sn_lower_bound_general}.  All graphs throughout this section are assumed to be connected.

\subsection{Applications with \(k=1\).}  We start with the following result.  \begin{theorem}\label{theorem:product_gon}
\begin{itemize}  Let \(G\) and \(H\) be connected graphs.  
\item[(i)]If \(G\) is a tree and \(H\) is a graph with \(\frac{|V(H)|}{\lambda(H)}\leq |V(G)|\), then \(\textrm{sn}(G\square H)=\textrm{gon}(G\square H)=|V(H)|\).
\item[(ii)]  If \(G\) and \(H\) are graphs with \(\gon(H)=\lambda(H)\) and \(|V(G)|\leq \frac{|V(H)|}{\lambda(H)}\), then \(\textrm{sn}(G\square H)=\textrm{gon}(G\square H)=|V(G)|\cdot \lambda(H)\).
\end{itemize}
\end{theorem}

\begin{proof}

From Corollary \ref{corollary:k=1} we have
\[\min(|V(H)|,|V(G)|\cdot \lambda(H))\leq\textrm{sn}(G\square H).\]

For (i), assume \(G\) is a tree and \(\frac{|V(H)|}{\lambda(H)}\leq |V(G)|\).  
First note that \(|V(H)|=\min(|V(H)|,|V(G)|\cdot \lambda(H))\leq\textrm{sn}(G\square H)\), and that \(\textrm{gon}(G\square H)\leq |V(H)|\cdot\gon(G)=|V(H)|\).  Since the lower bound on scramble number equals the upper bound on gonality, we may conclude that \(\textrm{sn}(G\square H)=\textrm{gon}(G\square H)=|V(H)|\).

 For (ii), assume \(\gon(H)=\lambda(H)\) and \(|V(G)|\leq \frac{|V(H)|}{\lambda(H)}\). We have \(|V(G)|\cdot \lambda(H)=\min(|V(H)|,|V(G)|\cdot \lambda(H))\leq \textrm{sn}(G\square H)\).  On the other hand, we have \(\textrm{gon}(G\square H)\leq |V(G)|\cdot \textrm{gon}(H)=|V(G)|\cdot \lambda(H)\). Since the lower bound on scramble number equals the upper bound on gonality, we may conclude that \(\textrm{sn}(G\square H)=\textrm{gon}(G\square H)=|V(G)|\cdot \lambda(H)\).

\end{proof}

We now apply this theorem over the course of several corollaries.

\begin{corollary}\label{corollary:tree_and_gon=lambda}
Let \(G\) be a tree and \(H\) a \(k\)-edge-connected graph of gonality \(k\).  Then
\[\textrm{sn}(G\square H)=\textrm{gon}(G\square H)=\textrm{min}(|V(H)|,k|V(G)|).\]
\end{corollary}

\begin{proof}  If \(\textrm{min}(|V(H)|,k|V(G)|)=|V(H)|\), then all the assumptions of Theorem \ref{theorem:product_gon}(i) are satisfied, and so \(\textrm{sn}(G\square H)=\textrm{gon}(G\square H)=|V(H)|\).  If \(\textrm{min}(|V(H)|,k|V(G)|)=k|V(G)|\), then all the assumptions of Theorem \ref{theorem:product_gon}(ii) are satisfied, and so \(\textrm{sn}(G\square H)=\textrm{gon}(G\square H)=k|V(G)|\).  
\end{proof}

Note that \(H\) can be any of the graphs from Example \ref{example:lambda=gon}, including trees, complete multipartite graphs, bridgeless hyperelliptic graphs, and simple bridgeless trivalent graphs of gonality three.  (We remark that the gonality of \(G\square H\) with \(G\) a tree and \(H\) either a tree or a complete graph was already known by \cite[Propositions 11, 12]{gonality_product}.)

For integers \(\ell,m,n\geq 2\), we let \(G_{m,n}:=P_m\square P_n\) denote the \(m\times n\) two-dimensional grid graph, and we let \(G_{\ell,m,n}:=P_\ell\square P_m\square P_n\) denote the \(\ell\times m\times n\) three-dimensional grid graph. 

\begin{corollary}
If \(\ell,m,n\geq 2\) with \(\ell\geq mn/2\), then
\[\textrm{sn}(G_{\ell,m,n})=\textrm{gon}(G_{\ell,m,n})=mn.\]
\end{corollary}

We remark that this gonality is conjectured to hold for all values of \(\ell,m,n\); see \cite[Conjecture 4.6]{db}.

\begin{proof}
Write \(G_{\ell,m,n}=P_{\ell}\square G_{m,n}\) as a product of the path on \(\ell\)-vertices \(G=P_\ell\) and the \(m\times n\) grid graph \(H=G_{m,n}\).  Note that the first graph is a tree, and the second graph has edge-connectivity 2, so we have \(|V(H)|/\lambda(H)=mn/2\leq \ell= |V(G)|\).  This allows us to conclude by Theorem \ref{theorem:product_gon}(i) that the gonality of \(G_{\ell,m,n}\) is equal to \(|V(H)|=m n\).  
\end{proof}

\begin{corollary}
Let \(T\) be a tree on \(m\geq 2\) vertices and \(G\) be a graph on \(n\leq m\) vertices.  Then
\[\textrm{sn}(G\square K_\ell \square T)=\textrm{gon}(G\square K_\ell \square T)=\ell n.\]
\end{corollary}

\begin{proof}
Let \(H=K_\ell\square T\).  First we claim that that \(\lambda(H)=\textrm{gon}(H)=\ell\).  We certainly have \(\lambda(H)\leq\textrm{gon}(H)\), and by \cite[Proposition 12]{gonality_product} we know \(\textrm{gon}(H)=\ell\).  To see that \(\lambda(H)\geq \ell\), we use \cite[Theorem 1]{connectivity_products} to compute the vertex-connectivity \(\kappa(K_\ell \square T)\) to be
\begin{align*}
\kappa(K_\ell \square T)=&\min(\kappa(K_\ell)|V(T)|,\kappa(T)|V(K_\ell)|,\delta(K_\ell)+\delta(T))
\\=&\min((\ell-1)\cdot m,1\cdot \ell,\ell-1+1)=\ell.
\end{align*}
We thus have \(\ell=\kappa(K_\ell \square T)\leq \lambda(K_\ell \square T)\leq \textrm{gon}(K_\ell \square T)=\ell\), so all these numbers are equal to \(\ell\).

Since \(\lambda(H)=\textrm{gon}(H)=\ell\), we may apply Theorem \ref{theorem:product_gon}(ii) whenever taking the product of \(H\) with a graph with at most \(|V(H)|/\lambda(H)=\ell m/\ell=m\) vertices.  Since \(n\leq m\), we conclude that
\[\textrm{sn}(G\square H)=\textrm{gon}(G\square H)=|V(G)|\cdot \lambda(H)=\ell n.\]
\end{proof}

A special case of this is any product of the form \(K_\ell\square T_1\square T_2\) where \(T_1\) and \(T_2\) are trees on \(m\) and \(n\) vertices (if \(n\leq m\), then we use \(G=T_2\) and \(T=T_1\)). In this case we find 
\(\textrm{sn}(K_\ell\square T_1\square T_2)=\textrm{gon}(K_\ell\square T_1\square T_2)=\ell\cdot \min(m,n)\)  Note that since the grid graph \(G_{m,n}\) is the product of two trees, this means
\(\textrm{sn}(K_\ell\square G_{m,n})=\textrm{gon}(K_\ell\square G_{m,n})=\ell\cdot \min(m,n)\).

\begin{corollary}\label{corollary:bridgeless_hyperelliptic}
If \(H\) is a bridgeless hyperelliptic graph on \(m\) vertices and \(G\) is a graph with \(n\leq m/2\) vertices, then \(\textrm{sn}(G\square H)=\textrm{gon}(G\square H)=2n\).
\end{corollary}

\begin{proof}
By assumption, \(\lambda(H)\geq 2\). Recall that as \(H\) is hyperelliptic, we have \(\gon(H)=2\) and \(V(H)\geq 3\). Since \(\min(\lambda(H),|V(H)|)\leq\gon(H)=2\), it follows that  \(\lambda(H)\leq 2\), so in fact \(\lambda(H)=\textrm{gon}(H)=2\).  This lets us apply Theorem \ref{theorem:product_gon}(ii) whenever we take the product of \(H\) with a graph with \(n\) vertices where \(n\) is at most \(|V(H)|/\lambda(H)=m/2\), which is precisely the assumed set-up.  This gives us a scramble number and a gonality of \(n\cdot \lambda(H)=2n\), as claimed.
\end{proof}

A familiar example of such a graph \(H\) is the cycle graph \(C_m\).  Thus the product of any graph with a sufficiently large cycle has known gonality.

\begin{corollary}
Let \(K_{m,n}\) denote the complete bipartite graph on \(m,n\) vertices with \(m\leq n\).  If \(G\) is a graph on \(\ell\) vertices with \(\ell\leq (m+n)/m\), then
\[\textrm{sn}(G\square K_{m,n})=\textrm{gon}(G\square K_{m,n})=\ell m.\]
\end{corollary}

\begin{proof}
The edge-connectivity of \(K_{m,n}\) is \(\min(m,n)=m\), as is its gonality by \cite[Example 4.3]{treewidth}.  Thus allows us to apply Theorem \ref{theorem:product_gon}(ii) to \(G\square K_{m,n}\) whenever \(G\) has at most \((m+n)/m\) vertices, giving us the claimed formula.
\end{proof}

\subsection{Applications with \(k=2\).}

\begin{theorem}\label{theorem:gonality_two_cases}
Let \(G\) and \(H\) be connected graphs with \(\kappa(G)\geq 2\).

\begin{itemize}
    
    \item[(i)]   If \(\textrm{gon}(G)=2\), \(\frac{|V(H)|}{\lambda(H)}\leq \frac{1}{2}|V(G)|\), and \({|V(H)|}\leq \frac{1}{2}|V(G)|\lambda(H)+(\delta(G)-\lambda(H))\), then 
    \[\textrm{sn}(G\square H)=\textrm{gon}(G\square H)=2|V(H)|.\]
    
    \item[(ii)]  If \(\lambda(H)=\textrm{gon}(H)\), \(\frac{2|V(H)|}{\lambda(H)}\geq |V(G)|\), and \(\lambda(H)\leq\delta(G)\), then
     \[\textrm{sn}(G\square H)=\textrm{gon}(G\square H)=|V(G)|\lambda(H).\]
\end{itemize}

\end{theorem}

\begin{proof}
By Proposition \ref{proposition:k=2}, we know in both cases that \[\textrm{sn}(G\square H)\geq \min\left(2|V(H)|,|V(G)|\lambda(H),(|V(G)|-2)\lambda(H)+2\delta(G)\right). \]


For (i), our assumptions give us that the minimum of the three terms is \(2|V(H)|\), so this number is a lower bound on scramble number.    
We also have \(\textrm{sn}(G\square H)\leq\textrm{gon}(G\square H)\leq \textrm{gon}(G)\cdot |V(H)|=2|V(H)|\).  Since the upper and lower bounds agree, we have the claimed equality.

For (ii), since \(\lambda(H)\leq\delta(G)\), we have \(|V(G)|\lambda(H)\leq (|V(G)|-2)\lambda(H)+2\delta(G)\), so
\[\textrm{sn}(G\square H)\geq \min\left(2|V(H)|,|V(G)|\lambda(H)\right). \]
From here we get a lower bound of \(|V(G)|\lambda(H)\) and an upper bound of \(|V(G)|\textrm{gon}(H)\); these are equal, giving us the claimed result.
\end{proof}

Here are several applications of this result.

\begin{corollary}
If \(m,n\geq 2\), then
\[\textrm{gon}(C_m\square K_n)=\min(2n,m(n-1)).\]
If further \(m\geq 4\), then this is also equal to \(\textrm{sn}(C_m\square K_n)\).
\end{corollary}

\begin{proof}
If \(n=2\) then we have the result since we have the product of a cycle and a tree; assume for the remainder of the proof that \(n\geq 3\). For the moment we will also assume \(m\geq 4\); the remaining cases will be handled at the end.  Note that for \(m\geq 4\) we have \(\min(2n,m(n-1))=2n\).

If \(n\geq m\), set \(G=C_m\) and \(H=K_n\).  Then we have that \(\kappa(G)=\textrm{gon}(G)=2\); that \(\frac{|V(H)|}{\lambda(H)}=\frac{n}{n-1}\leq\frac{3}{2}\leq \frac{1}{2}|V(G)|\); and that \(\frac{1}{2}|V(G)|\lambda(H)+\delta(G)-\lambda(H)=\frac{m(n-1)}{2}+2-(n-1)=\frac{(m-2)(n-1)}{2}+2\geq(n-1)+2 >n=|V(H)|\) since \(m-2\geq 2\).  This lets us apply  Theorem \ref{theorem:gonality_two_cases}(i) to conclude
\[\textrm{sn}(C_m\square K_n)=\textrm{gon}(C_m\square K_n)=2|V(H)|=2n.\]

If \(n\leq m\), set  \(G=K_n\) and \(H=C_m\).  We have \(\kappa(G)=n-1\geq 2\); \(\lambda(H)=\textrm{gon}(G)=2\); \(\frac{2|V(H)|}{\lambda(H)}=|V(H)|=m\geq n=|V(H)|\); and \(\lambda(H)=2\leq n-1=\delta(G)\).  This lets us apply Theorem \ref{theorem:gonality_two_cases}(ii) to conclude
\[\textrm{sn}(C_m\square K_n)=\textrm{gon}(C_m\square K_n)=|V(G)|\lambda(H)=2n.\]

We can now handle those cases where \(m<4\).  If \(m=3\) then the cycle graph is the complete graph \(K_3\), and we already have this claimed gonality by \cite[Theorem 17]{gonality_product}.  Finally, suppose \(m=2\), so that \(\min(2n,m(n-1))=2n-2\).  Certainly we have \(\textrm{gon}(C_2\square K_n)\leq 2n-2\). Suppose for the sake of contradiction that \(\textrm{gon}(C_2\square K_n)<2n-2\), and let \(D\) be an effective positive rank divisor of degree \(2n-3\) on \(C_2\square K_n\).  By  Lemma \ref{lemma:spencer} we may assume that \(D\) has fewer than valence-many chips on each vertex, so that every vertex has at most \((n+1)-1=n\) chips on it. Let \(V(C_2)=\{u,v\}\). At least one of the canonical copies of \(K_n\) has at most \(n-2\) chips; say it is \(u\square K_n\).  Choose \(q\) a vertex in \(u\square K_n\) with no chips on it, and run Dhar's burning algorithm on \(D-(q)\).  By Lemma \ref{lemma:complete_dhars}, all of \(u\square K_n\) will burn.  Since \(D\) has positive rank, we know by Lemma \ref{lemma:dhars} that not all of \(v\square K_n\) burns; say there are \(k\) unburned vertices when the burning process stabilizes.  Note that \(k\geq 2\), since if \(k=1\) then a vertex must have valence many chips.  The total number of burning edges coming into those \(k\) vertices is then \(2k+k(n-k)\) (where \(2k\) come from the other canonical copy of \(K_n\), and \(k(n-k)\) come from the same canonical copy).  This expression is concave down as a function of \(k\), and so is minimized on the interval \(2\leq k\leq n\) at an endpoint.  At \(k=2\) we have \(2\cdot 2+2(n-2)=2n\) edges, and at \(k=n\) we have \(2n+n(n-n)=2n\) edges.  This means that these \(k\) vertices must have a total of at least \(2n\) chips, a contradiction to \(\deg(D)=2n-3\).  We conclude that \(\textrm{gon}(C_2\square K_n)=2n-2\).
\end{proof}

\begin{corollary}\label{corollary:2-connected_hyp}
If \(G\) and \(H\) are \(2\)-connected hyperelliptic graphs, then
\[\textrm{sn}(G\square H)=\textrm{gon}(G\square H)=\min(2|V(G)|,2|V(H)|).\]
\end{corollary}

\begin{proof}
Because both \(G\) and \(H\) are \(2\)-connected, we may assume without loss of generality that  \(|V(H)|\leq |V(G)|\), so \(\min(2|V(G)|,2|V(H)|)=2|V(H)|\).  Since \(G\) is \(2\)-connected and hyperelliptic and \(\lambda(G)\geq \kappa(G)\), we have \(\kappa(G)=\lambda(G)=\gon(G)=2\), as argued in the proof of Corollary \ref{corollary:bridgeless_hyperelliptic}; the same holds for \(H\).
We also have \(\delta(G)\geq \kappa(G)\geq 2=\lambda(H)\). Since \(\textrm{gon}(G)=2\) and \(\frac{2|V(H)|}{\lambda(H)}=|V(H)|\leq |V(G)|\), we may conclude by Theorem \ref{theorem:gonality_two_cases}(i) that
\[\textrm{sn}(G\square H)=\textrm{gon}(G\square H)=2|V(H)|=\min(2|V(G)|,2|V(H)|).\]
\end{proof}

This is a natural generalization of \cite[Proposition 5.4]{harp2020new}, which is the same result in the special case that \(G\) and \(H\) are both cycle graphs.  Some nice examples of graphs we get from this result are \(4\)-dimensional glued grid graphs of the form \(G_{2,2,m,n}\), since we can write \(G_{2,2,m,n}=G_{2,m}\square G_{2,n}\), and \(G_{2,m}\) and \(G_{2,n}\) satisfy the hypotheses of the corollary.

\subsection{Applications with \(k\geq 3\).}

When \(k\geq 3\), the term \((|V(G)|-2k+2)\lambda(H)+2\lambda(G)\) becomes cumbersome to work with.  We prove several results whose hypotheses ensure that this term is not the minimum of the three terms.

\begin{theorem}\label{theorem:gonality_three_cases}  Assume \(G\) and \(H\) are graphs with \(k\leq\kappa(G)\), \(|V(G)|\geq 2k-1\), and \(\lambda(G)\geq (k-1)\lambda(H)\).
\begin{itemize}
    \item[(i)]  If \(k|V(H)|\leq |V(G)|\lambda(H)\) and \(\textrm{gon}(G)=k\), then \(\textrm{sn}(G\square H)=\textrm{gon}(G\square H)=k|V(H)|\).
    \item[(ii)]  If \(|V(G)|\lambda(H)\leq k|V(H)|\) and \(\textrm{gon}(H)=\lambda(H)\), then \(\textrm{sn}(G\square H)=\textrm{gon}(G\square H)=|V(G)|\lambda(H).\)
\end{itemize}
\end{theorem}

\begin{proof}
We have \[\textrm{sn}(G\square H)\geq \min\left(k|V(H)|,|V(G)|\lambda(H),(|V(G)|-2k+2)\lambda(H)+2\lambda(G)\right), \]
which we can rewrite as
\[\textrm{sn}(G\square H)\geq \min\left(k|V(H)|,|V(G)|\lambda(H),|V(G)|\lambda(H)+2(\lambda(G)-(k-1)\lambda(H))\right). \]
Since \(\lambda(G)\geq (k-1)\lambda(H)\), we have the simpler bound of 
\[\textrm{sn}(G\square H)\geq \min\left(k|V(H)|,|V(G)|\lambda(H)\right). \]
The claims in cases (i) and (ii) then come from the upper bound \[\textrm{gon}(G\square H)\leq\min(|V(G)|\gon(H),|V(H)|\gon(G))\] and the assumed hypotheses.
\end{proof}

We remark that for a simple graph \(G\), in order for \(\gon(G)=k\leq\kappa(G)\) as assumed in case (i), then \(\lambda(G)=k\) as well, so for \(k\geq 3\) we must have \(\lambda(H)=1\) for the assumption of \(\lambda(G)\geq (k-1)\lambda(H)\) to hold.  We summarize this in the following result.

\begin{corollary}  Assume \(G\) and \(H\) are graphs with \(k=\kappa(G)=\gon(G)\), \(|V(G)|\geq 2k-1\), \(\lambda(H)=1\), and \(|V(H)|\leq |V(G)|/k\).  Then \(\textrm{sn}(G\square H)=\textrm{gon}(G\square H)=k|V(H)|\).
\end{corollary}

A concrete application of Theorem \ref{theorem:gonality_three_cases} is for \(3\)-dimensional toroidal grid graphs which are ``sufficiently oblong''.

\begin{corollary}
Let \(\ell,m,n\geq 2\) with \(\frac{2}{3}\ell m\leq n\) and  \(\max(\ell,m)\geq 3\).  Then
\[\sn(C_\ell\square C_m\square C_n)=\gon(C_\ell\square C_m\square C_n)=2\ell m.\]
\end{corollary}

\begin{proof}  Let \(G=C_\ell\square C_m\) and \(H=C_n\).  Note that \(\kappa(G)=4>3\), that \(\kappa(G)=\ell m\geq 6>2\cdot 3-1\), and \(\lambda(G)=4=2\cdot 2=(3-1)\cdot\lambda(H)\), so the starting hypotheses of Theorem \ref{theorem:gonality_three_cases} are satisfied with \(k=3\).  Moreover, \(|V(G)|\lambda(H)=2\ell m\leq 3 n=3|V(H)|\), and \(\gon(H)=2=\lambda(H)\).  This allows us to apply Theorem \ref{theorem:gonality_three_cases}(ii) to obtain the claimed formula.  
\end{proof}

We remark that a similar result could have been obtained from Theorem \ref{theorem:gonality_two_cases}; however, it would have required the stronger assumption that \(\ell m\leq n\).

\begin{theorem}\label{theorem:many_equal_k}  Assume that \(G\) and \(H\) satisfy \[\kappa(G)=\lambda(G)=\textrm{gon}(G)=k\leq \lambda(H)\]
where \(|V(H)|\leq |V(G)|-2k+4\) and \(|V(G)|\geq 2k-1\).  Then
\[\textrm{sn}(G\square H)=\textrm{gon}(G\square H)=k\cdot |V(H)|.\]
\end{theorem}

\begin{proof}
If we further assume that \(\lambda(H)=k\), then Theorem \ref{theorem:sn_lower_bound_general} gives us
\begin{align*}
    \textrm{sn}(G\square H)\geq& \min\left(k|V(H)|,|V(G)|\lambda(H),(|V(G)|-2k+2)\lambda(H)+2\lambda(G)\right)
    \\=&\min\left(k|V(H)|,k|V(G)|,(|V(G)|-2k+2)k+2k\right)
    \\=&k\cdot \min\left(|V(H)|,|V(G)|,|V(G)|-2k+4\right)
    \\=&k\cdot|V(H)|,
\end{align*}where the final equality comes from the assumption \(|V(H)|\leq |V(G)|-2k+4\). For \(\lambda(H)\geq k\),  the term \(k\cdot |V(H)|\) will still be the minimum, since the other terms will not have been decreased.  On the other hand, \(\textrm{gon}(G\square H)\leq |V(H)|\textrm{gon}(G)=k\cdot |V(H)|\).  The equality of our upper and lower bounds gives the claimed result.
\end{proof}

\begin{corollary}
If \(T\) is a tree with \(m\geq 2\) vertices and \(k(m-2)+4\geq \ell>k\), then \[\textrm{sn}(K_{\ell}\square K_{k}\square T)=\textrm{gon}(K_{\ell}\square K_{k}\square T)=k\ell.\]
\end{corollary}

\begin{proof}
Let \(G=K_{k}\square T\) and \(H=K_{\ell}\).  We then have that \(k=\kappa(G)=\lambda(G)=\gon(G)\leq \ell-1=\lambda(H)\), that \(|V(H)|=\ell\leq km-2k+4=|V(G)|-2k+4\), and that \(|V(G)|=km\geq 2k-1\).  This allows us to apply Theorem \ref{theorem:many_equal_k} to conclude the claimed result.
\end{proof}

Note that \(\ell=4\), \(k=3\), and \(m=2\) satisfies the hypotheses, so the gonality of the \(3\)-dimensional rook's graph \(K_4\square K_3\square K_2\) is \(12\) since \(K_2\) is a tree.  However, no larger value of \(\ell\) is allowed when \(m=2\).

\begin{corollary}
If \(n\leq \ell\leq k\), \(n\leq m\), and  \(k+\ell+n-4\leq m \), then
\[\textrm{sn}(K_{k,\ell}\square K_{m,n})=\textrm{gon}(K_{k,\ell}\square K_{m,n})=(k+\ell)n.\]
\end{corollary}

\begin{proof}
We have
\[\lambda(K_{k,\ell})=\ell\geq n=\kappa(K_{m,n})=\lambda(K_{m,n})=\textrm{gon}(K_{m,n})\]
where \(|V(K_{k,\ell})|=k+\ell\leq m-n+4=m+n-2n+4=|V(K_{m,n})|-2n+4\), allowing us to apply Theorem \ref{theorem:many_equal_k} to obtain the claimed result.
\end{proof}

\subsection{Open questions in product gonality}

All of the gonalities of product graphs that we have computed in this paper are ones where it turned out that we had the ``expected gonality'', equal to the upper bound of
\[\min(|V(G)|\gon(H),|V(H)|\gon(G)).\]Not all product graphs have this expected gonality; indeed, there exists a genus \(1\) graph \(G\) on three vertices with \(\lambda(G)=1\) such that \(\textrm{gon}(G\square G)\leq 5\), even though the ``expected'' gonality is \(6\) \cite[\S 1]{gonality_product}.  We summarize below some basic families of product graphs where we might hope to determine whether or not all gonalities are given by the upper bound.

\begin{itemize}
    \item \textbf{Tree product tree.}  These are known to have expected gonality by \cite[Proposition 11]{gonality_product}.
    
    \item \textbf{Tree product hyperelliptic.}  If the hyperelliptic graph \(H\) is \(2\)-edge-connected, then by Corollary \ref{corollary:tree_and_gon=lambda} we know that the product of it with a tree \(T\) has the expected gonality.  However, if the hyperelliptic graph is only \(1\)-edge-connected, the gonality of \(T\square H\) is in general open: we have an upper bound of \(\min(2|V(T)|,|V(H)|)\), and a lower bound of \(\min(|V(T)|,|V(H)|)\) from Corollary \ref{corollary:k=1}.
    
    \item \textbf{Hyperelliptic product hyperelliptic.}  The aforementioned counterexample does fall into the category of hyperelliptic product hyperelliptic, so not all product graphs of this form have the expected gonality.  However, by our Corollary \ref{corollary:2-connected_hyp}, we do have expected gonality if both hyperelliptic graphs are \(2\)-connected.  As the counterexample from \cite{gonality_product} had both graphs \(1\)-edge connected, we might ask whether we can always get expected gonality when both graphs are \(2\)-edge-connected, or when perhaps only one of them has an edge- or vertex-connectivity assumption.
\end{itemize}

Since all known examples of graphs \(G\square H\) with lower than expected gonality have both \(G\) and \(H\) graphs of positive genus, it is natural to pose the following question.

\begin{question}
Does there exist a tree \(T\) and a graph \(H\) where \(\textrm{gon}(T\square H)\) does not have the expected gonality of \(\min(|V(H)|,|V(T)|\gon(H))\)?
\end{question}

If the answer to this question is ``no'', then it would follow by induction that the gonality of the \(n\)-dimensional hypercube graph \(Q_n\) is \(2^{n-1}\) as conjectured in \cite[\S 4]{treewidth}.  Although we do not have a complete answer to this question, there are several cases handled earlier in this section where we now know that \(T\square H\) has the expected gonality:
\begin{itemize}
    \item When \(\lambda(H)=\textrm{gon}(H)\), by Corollary \ref{corollary:tree_and_gon=lambda}.
    \item  When \(|V(H)|/\lambda(H)\leq |V(T)|\), by Theorem \ref{theorem:product_gon}(i).
\end{itemize}

\begin{figure}[hbt]
    \centering
    \includegraphics{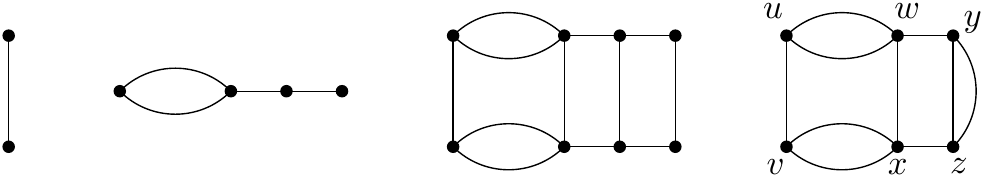}
    \caption{From left to right:  \(T=K_2\), \(H\), \(T\square H\), and \(J\)}
    \label{figure:product_example}
\end{figure}
We remark that scramble number will not be able to compute the gonality of \(T\square H\) for all choices of \(T\) and \(H\).  Consider for instance the graphs \(T=K_2\) and \(H\) pictured in Figure \ref{figure:product_example}, along with their product \(T\square H\) and a graph \(J\) obtained from \(T\square H\) by smoothing over \(2\)-valent vertices.  By \cite[Theorem 18]{connectivity_products}, we have \(\gon(T\square H)=4\).  Suppose for the sake of contradiction that \(T\square H\) has a scramble of order \(4\).  Then \(J\) also has a scramble \(\mathcal{S}\) of order \(4\) by \cite[Proposition 4.4]{harp2020new}.  There cannot be eggs completely contained in both of \(\{u,v,w,x\}\) and \(\{y,z\}\), since then \(A=\{y,z\}\) would be an egg-cut of size \(2\).  If all eggs intersect \(\{y,z\}\), then there must exist a hitting set of size \(2\); thus all eggs intersect \(\{u,v,w,x\}\).  However, \(\{v,w,x\}\) cannot be a hitting set, implying that \(\{u\}\) is an egg.  But then either \(A=\{u\}\) forms an egg-cut of size \(3\); or all eggs contain \(\{u\}\) and \(h(\mathcal{S})=1\).  All of these contradict \(||\mathcal{S}||=4\).  Thus \(\sn(T\square H)<\gon(T\square H)\).

\bibliographystyle{abbrv}

\begin{thebibliography}{10}

\bibitem{gonality_three}
Ivan Aidun, Frances Dean, Ralph Morrison, Teresa Yu, and Julie Yuan.
\newblock Graphs of gonality three.
\newblock {\em Algebr. Comb.}, 2(6):1197--1217, 2019.

\bibitem{gonality_sequences}
Ivan Aidun, Frances Dean, Ralph Morrison, Teresa Yu, and Julie Yuan.
\newblock Gonality sequences of graphs.
\newblock {\em{SIAM J. on Disc. Math.}} 35(2):814--839, 2021.

\bibitem{glued_grid_graphs}
Ivan Aidun, Frances Dean, Ralph Morrison, Teresa Yu, and Julie Yuan.
\newblock Treewidth and gonality of glued grid graphs.
\newblock {\em Discrete Appl. Math.}, 279:1--11, 2020.

\bibitem{gonality_product}
Ivan Aidun and Ralph Morrison.
\newblock On the gonality of cartesian products of graphs.
\newblock {\em Electron. J. Combin.}, 27(4), 2020.

\bibitem{spencer}
Spencer Backman.
\newblock Riemann-{R}och theory for graph orientations.
\newblock {\em Adv. Math.}, 309:655--691, 2017.

\bibitem{sandpile}
Per Bak, Chao Tang, and Kurt Wiesenfeld.
\newblock Self-organized criticality.
\newblock {\em Phys. Rev. A (3)}, 38(1):364--374, 1988.

\bibitem{baker_specialization}
Matthew Baker.
\newblock Specialization of linear systems from curves to graphs.
\newblock {\em Algebra Number Theory}, 2(6):613--653, 2008.
\newblock With an appendix by Brian Conrad.

\bibitem{bn}
Matthew Baker and Serguei Norine.
\newblock Riemann-{R}och and {A}bel-{J}acobi theory on a finite graph.
\newblock {\em Adv. Math.}, 215(2):766--788, 2007.

\bibitem{bn2}
Matthew Baker and Serguei Norine.
\newblock Harmonic morphisms and hyperelliptic graphs.
\newblock {\em Int. Math. Res. Not. IMRN}, (15):2914--2955, 2009.

\bibitem{chip_firing}
Anders Bj\"{o}rner, L\'{a}szl\'{o} Lov\'{a}sz, and Peter~W. Shor.
\newblock Chip-firing games on graphs.
\newblock {\em European J. Combin.}, 12(4):283--291, 1991.

\bibitem{cdpr}
Filip Cools, Jan Draisma, Sam Payne, and Elina Robeva.
\newblock A tropical proof of the {B}rill-{N}oether theorem.
\newblock {\em Adv. Math.}, 230(2):759--776, 2012.

\bibitem{ckk}
Gunther Cornelissen, Fumiharu Kato, and Janne Kool.
\newblock A combinatorial {L}i-{Y}au inequality and rational points on curves.
\newblock {\em Math. Ann.}, 361(1-2):211--258, 2015.

\bibitem{gonality_random_graphs}
Andrew Deveau, David Jensen, Jenna Kainic, and Dan Mitropolsky.
\newblock Gonality of random graphs.
\newblock {\em Involve}, 9(4):715--720, 2016.

\bibitem{dhar}
Deepak Dhar.
\newblock Self-organized critical state of sandpile automaton models.
\newblock {\em Phys. Rev. Lett.}, 64(14):1613--1616, 1990.

\bibitem{algebraic_gonality}
David Eisenbud.
\newblock {\em The geometry of syzygies}, volume 229 of {\em Graduate Texts in
  Mathematics}.
\newblock Springer-Verlag, New York, 2005.
\newblock A second course in commutative algebra and algebraic geometry.

\bibitem{gk}
Andreas Gathmann and Michael Kerber.
\newblock A riemann-roch theorem in tropical geometry.
\newblock {\em Mathematische Zeitschrift}, 259:217–230, 2007.

\bibitem{graph_gonality_is_hard}
Dion Gijswijt, Harry Smit, and Marieke van~der Wegen.
\newblock Computing graph gonality is hard.
\newblock {\em Discrete Appl. Math.}, 287:134--149, 2020.

\bibitem{harp2020new}
Michael Harp, Elijah Jackson, David Jensen, and Noah Speeter.
\newblock A new lower bound on graph gonality, 2020.

\bibitem{rank_determining}
Ye~Luo.
\newblock Rank-determining sets of metric graphs.
\newblock {\em J. Combin. Theory Ser. A}, 118(6):1775--1793, 2011.

\bibitem{Menger1927}
Karl Menger.
\newblock Zur allgemeinen kurventheorie.
\newblock {\em Fundamenta Mathematicae}, 10(1):96--115, 1927.

\bibitem{mz}
Grigory Mikhalkin and Ilia Zharkov.
\newblock Tropical curves, their {J}acobians and theta functions.
\newblock In {\em Curves and abelian varieties}, volume 465 of {\em Contemp.
  Math.}, pages 203--230. Amer. Math. Soc., Providence, RI, 2008.

\bibitem{treewidth_original}
Neil Robertson and P.~D. Seymour.
\newblock Graph minors. {III}. {P}lanar tree-width.
\newblock {\em J. Combin. Theory Ser. B}, 36(1):49--64, 1984.

\bibitem{seymour-thomas}
P.~D. Seymour and Robin Thomas.
\newblock Graph searching and a min-max theorem for tree-width.
\newblock {\em J. Combin. Theory Ser. B}, 58(1):22--33, 1993.

\bibitem{db}
Josse van Dobben~de Bruyn.
\newblock Reduced divisors and gonality in finite graphs, 2012.

\bibitem{treewidth}
Josse van Dobben~de Bruyn and Dion Gijswijt.
\newblock Treewidth is a lower bound on graph gonality.
\newblock {\em Algebr. Comb.}, 3(4):941--953, 2020.

\bibitem{connectivity_products}
Simon \v{S}pacapan.
\newblock Connectivity of {C}artesian products of graphs.
\newblock {\em Appl. Math. Lett.}, 21(7):682--685, 2008.

\end{thebibliography}

\end{document}